\documentclass{article}
\usepackage{arxiv}

\usepackage[utf8]{inputenc}
\usepackage[T1]{fontenc}
\usepackage{hyperref}
\hypersetup{
  colorlinks=true,
  linkcolor=blue,
  citecolor=blue,
  urlcolor=blue,
  pdftitle={Noncommutative Anisotropic Diffusion in Hilbert Space. I. The Consistent A-Geometry, Mosco Stability, and the Weak Bridge},
  pdfauthor={E. Yu. Shchetinin and A. A. Shevchuk and S. I. Salpagarov}
}
\usepackage{url}
\usepackage{amsmath,amssymb,amsthm,amsfonts}
\usepackage{nicefrac}
\usepackage{microtype}
\usepackage{graphicx}

\newtheorem{theorem}{Theorem}[section]
\newtheorem{lemma}[theorem]{Lemma}
\newtheorem{proposition}[theorem]{Proposition}
\newtheorem{corollary}[theorem]{Corollary}
\newtheorem{assumption}[theorem]{Assumption}
\newtheorem{definition}[theorem]{Definition}
\newtheorem{remark}[theorem]{Remark}

\newcommand{\Hh}{\mathcal H}
\newcommand{\LL}{\mathcal L}
\newcommand{\LLone}{\mathcal L_1}
\newcommand{\LLtwo}{\mathcal L_2}
\newcommand{\E}{\mathbb E}

\newcommand{\Law}{\operatorname{Law}}
\newcommand{\Tr}{\operatorname{Tr}}
\newcommand{\Ent}{\mathrm H}
\newcommand{\Eval}{\mathcal L_{\mathrm{val}}}

\newcommand{\Asf}{\mathsf A}

\newcommand{\Ran}{\operatorname{Ran}}
\newcommand{\Dom}{\operatorname{Dom}}
\newcommand{\FC}{\mathcal{FC}_b^\infty}

\newcommand{\subjclass}[2][2020]{%
  \par\addvspace\medskipamount
  {\small\textit{2020 Mathematics Subject Classification.} #2\par}}

\title{Noncommutative Anisotropic Diffusion in Hilbert Space.\\
       I. The Consistent \(A\)-Geometry, Mosco Stability, and the Weak Bridge}

\author{%
  E.\,Yu.~Shchetinin \\
  Sevastopol State University \\
  Sevastopol, Russia \\
  \texttt{riviera-molto@mail.ru}
  \And
  A.\,A.~Shevchuk \\
  Sevastopol State University \\
  Sevastopol, Russia \\
  \texttt{andreiluck11@yandex.ru}
  \AND
  S.\,I.~Salpagarov \\
  RUDN University \\
  Moscow, Russia \\
  \texttt{salpagarov-si@rudn.ru}
}

\begin{document}
\maketitle
\nocite{DaPratoZabczyk,PrevotRockner,BogachevGaussian,Kuo,BogachevKrylovRockner,AmbrosioSavareZambotti,MaRockner,FukushimaOshimaTakeda,Mosco,KuwaeShioya,BakryGentilLedoux,HolleyStroock,Ledoux,OttoVillani,Villani,BoucheronLugosiMassart,Wainwright,Hyvarinen,Vincent,Song,Allaire,CioranescuDonato,Weickert,LimYoon2023,Pidstrigach2024,FranzeseMichiardi2025,Hagemann2025,DeBortoli2022,DeWitt2020,AnthonyBartlett,vanDerVaartWellner}

\begin{abstract}
This first part of the series builds the analytic layer of noncommutative
anisotropic diffusion in a separable Hilbert space. Let
\(\mu_0=\mathcal N(0,Q)\) be the reference Gaussian measure, with
\(Q\in\LLone(\Hh)\), and let \(D(x)\) be a positive, state-dependent
anisotropy. We do not assume that \([D(x),Q]=0\). Consequently, for the
forward SDE with
\[
        \sigma(x)=D(x)^{1/2}Q^{1/2}
\]
the correct energy form is given not by the expression
\(\langle D\nabla u,\nabla v\rangle\) but by the consistent form
\[
        \Gamma_A(u,v)=
        \langle Q^{1/2}D(x)^{1/2}\nabla u,
        Q^{1/2}D(x)^{1/2}\nabla v\rangle .
\]
We prove closability of the form, well-posedness of the forward dynamics,
Galerkin convergence, stability of the \(A\)-LSI under a Mosco limit, the
chain rule for relative entropy, and a general weak-bridge theorem. The main
result of Part~I is a functional-analytic theorem: if \(A\)-consistency, a
uniform \(A\)-LSI, and representability of the right-hand side of the backward
weak form in the negative energy space all hold, then a backward weak drift
\(v=\Asf\nabla\Phi\) exists and the basic entropy dissipation estimate holds. In
addition, we single out a three-dimensional tensor class of anisotropies, formulate a
condition for the absence of diffusion degeneracy, and obtain a rate estimate
for the homogenization limit, first on cylindrical subspaces and then on
compact-tail classes, which yields strong resolvent convergence and
convergence of the forward SDEs. The statistical closure, an independent
isotropic benchmark, and an approximation theorem for \(A\)-adapted networks
are treated in Part~II.
\end{abstract}

\keywords{noncommutative anisotropic diffusion \and Gaussian measure \and
  Dirichlet form \and logarithmic Sobolev inequality \and Mosco convergence
  \and relative entropy \and weak bridge \and homogenization}

\subjclass[2020]{60H15, 31C25, 47D07, 35B27, 28C20}

\section{Introduction}

Infinite-dimensional diffusion models require the joint treatment of three
structures: the stochastic dynamics, the geometry of the reference measure,
and the energy form
\cite{DaPratoZabczyk,PrevotRockner,BogachevGaussian}. In the commutative
situation, when \(D\) and \(Q\) are diagonalized in the same basis, the
Dirichlet form is usually written as \(\int\langle D\nabla u,\nabla v\rangle\,d\mu_0\).
The present work treats the noncommutative case, in which this reduction is
no longer valid.

An important difference from the work of Pidstrigach--Marzouk--Reich--Wang and
Hagemann--Mildenberger--Ruthotto--Steidl--Yang
\cite{Pidstrigach2024,Hagemann2025} is as follows. In those works the
infinite-dimensional setting is built around well-posedness of score-based
dynamics on function spaces and around discretization-invariant
approximations. In the present part the central object is not the generative
scheme itself but the noncommutative operator geometry: the noise has the form
\(\sigma(x)=D(x)^{1/2}Q^{1/2}\), so the carr\'e du champ is determined by the
operator \(\sigma(x)^*=Q^{1/2}D(x)^{1/2}\). When \([D(x),Q]\ne0\) this does not
reduce to a mere change of norm, as in the known results; one needs a separate
closability of the consistent form, a separate \(A\)-LSI, and a separate
duality for the weak bridge.

A further difference from
\cite{Pidstrigach2024,Hagemann2025} concerns the role of the operator \(Q\).
In the standard functional diffusion setting, \(Q\) usually fixes the
smoothness of the noise and the discretization invariance. In our case \(Q\)
also enters the energy geometry itself: the weak bridge, the dissipation, and
the logarithmic Sobolev inequality are all formulated through
\(Q^{1/2}D(x)^{1/2}\). Hence, even with the same forward noise regularity, two
models with a different operator order have different entropy
dissipations. This feature is not covered by isotropic or commutative theories
of score-based dynamics.

The aim of Part~I is to separate the analytic problem from the statistical
one. In contrast to current infinite-dimensional diffusion models
\cite{LimYoon2023,Pidstrigach2024,FranzeseMichiardi2025}, here the main object
is not only well-posedness of the forward evolution in Hilbert space, but also
the matching of the noncommutative anisotropy \(D(x)\) with the Gaussian
geometry \(Q\). We do not introduce neural-network classes, validation errors,
or lower bounds here. These questions are deferred to Part~II. Part~I
proves the basic theorem on which the score approximation and the applications
subsequently rely.

\paragraph{Terminology.}
The term \emph{score} denotes the logarithmic gradient of the density:
\(s_\rho=\nabla\log\rho\). The term ``weak bridge'' denotes the representation
of the backward weak evolution through a field \(v=\Asf\nabla\Phi\) in the
negative energy space. The word ``consistent'' refers to the order of
operators \(\sigma^*=Q^{1/2}D^{1/2}\), which determines the carr\'e du champ.

\section{A Detailed Comparison with Infinite-Dimensional Diffusion Models}
\label{sec:comparison-functional-diffusions}

Current work on infinite-dimensional diffusion models can be divided roughly
into two directions. The first studies the consistency of discretizations and
the existence of a limiting dynamics on function spaces. The second
investigates score approximation and the stability of learned backward fields.
The present part differs from both directions in that the main difficulty is
not only the passage to infinite dimension, but also the noncommutativity of
two operator structures: the covariance \(Q\) and the anisotropy \(D(x)\).

In the setting of Pidstrigach et al.\ \cite{Pidstrigach2024} the central
question is: how should one define a diffusion model on a function space so
that it does not depend on a finite-dimensional grid? Here, by contrast, the
forward noise is already given in Hilbert space, but its covariance has the
noncommutative structure \(D(x)^{1/2}QD(x)^{1/2}\). Hence, even when a strong
solution of the SDE exists, a separate problem remains: which Dirichlet form
corresponds to this noise, and which dissipation should enter the entropy
inequality. It is precisely this problem that is solved by introducing the
\(A\)-geometry.

In the work of Hagemann--Mildenberger--Ruthotto--Steidl--Yang
\cite{Hagemann2025} the main emphasis is on score-based models on function
spaces and on the approximation properties of the score field. In the present
Part~I the score field is not yet used: we prove an analytic reduction showing
that the entire statistical error in Part~II enters via a single
functional \(\Eval\). This separation is essential for the proof structure of
the series: well-posedness of the weak bridge and of the chain rule does not
depend on the chosen class of neural networks, and the statistical estimates
of Part~II need not re-prove closability of the form.

Thus the contribution of Part~I is not a repetition of existing Hilbert-space
diffusion frameworks, but the construction of a noncommutative energy layer
that can be used as the analytic basis for various classes of backward
approximations.

\section{Setup and the Consistent \(A\)-Geometry}

Let \(\Hh\) be a separable real Hilbert space, and let
\[
        \mu_0=\mathcal N(0,Q),\qquad
        Qe_k=q_ke_k,\qquad q_k>0,\qquad \sum_{k=1}^\infty q_k<\infty .
\]
We write \(\FC\) for the class of smooth cylindrical functions.

\begin{assumption}[Anisotropy]
\label{ass:A1}
For each \(x\in\Hh\) the operator \(D(x)\) is self-adjoint, positive, and
bounded. There exist \(0<d_-\le d_+<\infty\) such that
\[
        d_-I\le D(x)\le d_+I .
\]
The map \(x\mapsto D(x)^{1/2}\) is globally Lipschitz in the operator norm.
\end{assumption}

The forward noise is given by the operator
\[
        \sigma(x)=D(x)^{1/2}Q^{1/2}.
\]
Then
\[
        \sigma(x)^*=Q^{1/2}D(x)^{1/2}.
\]
It is precisely this order that determines the energy form.

We also write
\[
        \Asf(x):=D(x)^{1/2}QD(x)^{1/2}.
\]
This field of operators depends on \(x\); hence all expressions of the form
\(\Asf^{-1}\), \(\Asf\nabla\Phi\), and \(\|\,\cdot\,\|_{\Asf^{-1}}\)
appearing below are understood pointwise and are then integrated against the
appropriate measure.

\begin{definition}[The consistent \(A\)-form]
For \(u,v\in\FC\) set
\[
        \Gamma_A(u,v)(x)=
        \left\langle Q^{1/2}D(x)^{1/2}\nabla u(x),
        Q^{1/2}D(x)^{1/2}\nabla v(x)\right\rangle
\]
and
\[
        \mathcal E_A(u,v)=\int_{\Hh}\Gamma_A(u,v)\,d\mu_0 .
\]
\end{definition}

\begin{assumption}[\(A\)-compatibility]
\label{ass:Acompat}
There exist \(0<c_A\le C_A<\infty\) such that for all \(x,\xi\in\Hh\)
\[
        c_A\|Q^{1/2}\xi\|^2
        \le
        \|Q^{1/2}D(x)^{1/2}\xi\|^2
        \le
        C_A\|Q^{1/2}\xi\|^2 .
\]
\end{assumption}

\begin{theorem}[Closability of the consistent form]
\label{thm:closability}
Suppose Assumptions~\ref{ass:A1}--\ref{ass:Acompat} hold. Then the form
\((\mathcal E_A,\FC)\) is closable in \(L^2(\mu_0)\), its closure is a
symmetric Markovian Dirichlet form, and
\[
        c_A\int\|Q^{1/2}\nabla u\|^2\,d\mu_0
        \le
        \mathcal E_A(u,u)
        \le
        C_A\int\|Q^{1/2}\nabla u\|^2\,d\mu_0 .
\]
\end{theorem}

\begin{proof}
For cylindrical functions set \(\nabla_Q u=Q^{1/2}\nabla u\). We first check
closability. Let \(u_n\in\FC\), \(u_n\to0\) in \(L^2(\mu_0)\), and
\(\mathcal E_A(u_n-u_m,u_n-u_m)\to0\). By the lower \(A\)-compatibility bound,
\[
        c_A\|\nabla_Q(u_n-u_m)\|_{L^2(\mu_0;\Hh)}^2
        \le \mathcal E_A(u_n-u_m,u_n-u_m),
\]
so \(\nabla_Q u_n\) is Cauchy in \(L^2(\mu_0;\Hh)\). Denote its limit by
\(G\). For any basis vector \(e_k\) and any cylindrical \(\psi\), Gaussian
integration by parts gives
\[
        \int \langle \nabla_Q u_n,e_k\rangle \psi\,d\mu_0
        =-\int u_n\,\partial^{Q,*}_{e_k}\psi\,d\mu_0,
\]
where \(\partial^{Q,*}_{e_k}\psi\in L^2(\mu_0)\). The right-hand side tends to
zero, since \(u_n\to0\) in \(L^2\). Passing to the limit, we obtain
\(\int\langle G,e_k\rangle\psi\,d\mu_0=0\) for all \(k\) and all cylindrical
\(\psi\). Density of cylindrical functions in \(L^2(\mu_0)\) gives \(G=0\).
Now the upper \(A\)-compatibility bound yields
\[
        \mathcal E_A(u_n,u_n)\le C_A\|\nabla_Q u_n\|_{L^2(\mu_0;\Hh)}^2\to0.
\]
Hence the form is closable. The two-sided estimate in the statement follows at
once from Assumption~\ref{ass:Acompat}. The Markov property is checked first
for cylindrical truncations \(\eta\circ u\), where \(\eta\) is a
\(1\)-Lipschitz normal contraction:
\[
        \Gamma_A(\eta(u),\eta(u))\le \Gamma_A(u,u).
\]
After closure this inequality carries over to the domain of the closed form,
giving a symmetric Markovian Dirichlet form.
\end{proof}

\section{Forward Dynamics and Galerkin Convergence}

Consider the SDE
\[
        dX_t=b(t,X_t)\,dt+\sqrt{\beta(t)}D(X_t)^{1/2}Q^{1/2}\,dW_t .
\]

\begin{assumption}[Coefficients of the forward dynamics]
\label{ass:forward}
The function \(b(t,\cdot)\) is globally Lipschitz uniformly in \(t\in[0,T]\),
\(\beta\in C([0,T])\), \(0<\beta_-\le\beta(t)\le\beta_+\), and
\(\sigma(x)=D(x)^{1/2}Q^{1/2}\) is Lipschitz as a map into
\(\LLtwo(\Hh)\).
\end{assumption}

\begin{theorem}[Well-posedness of the forward SDE]
\label{thm:forward}
Under Assumption~\ref{ass:forward}, for each
\(X_0\in L^2(\Omega;\Hh)\) there is a unique strong solution
\(X\in L^2(\Omega;C([0,T];\Hh))\). Moreover,
\[
        \E\sup_{t\le T}\|X_t\|^2
        \le C_T^{\mathrm{fwd}}(1+\E\|X_0\|^2).
\]
\end{theorem}

\begin{proof}
We verify the conditions for strong solvability directly. From the Lipschitz
continuity of \(D^{1/2}\) in the operator norm, we have for
\(\sigma(x)=D(x)^{1/2}Q^{1/2}\)
\[
        \|\sigma(x)-\sigma(y)\|_{\LLtwo}
        \le
        \|D(x)^{1/2}-D(y)^{1/2}\|_{\LL}\|Q^{1/2}\|_{\LLtwo}
        \le L_{D^{1/2}}\|Q^{1/2}\|_{\LLtwo}\|x-y\|.
\]
Moreover,
\[
        \|\sigma(x)\|_{\LLtwo}^2
        =\Tr(D(x)^{1/2}QD(x)^{1/2})
        \le d_+\Tr Q.
\]
For the drift, by Assumption~\ref{ass:forward},
\[
        \|b(t,x)-b(t,y)\|\le L_b\|x-y\|,
        \qquad
        \|b(t,x)\|\le L_b\|x\|+B_0,
\]
where \(B_0=\sup_{t\le T}\|b(t,0)\|\). Thus the coefficients satisfy the
global Lipschitz and linear-growth conditions in \(\Hh\) and
\(\LLtwo(\Hh)\). Successive Picard approximations give a unique local
solution; It\^o's formula for \(\|X_t\|^2\), the
Burkholder--Davis--Gundy inequality, and Young's inequality give
\[
        \E\sup_{s\le t}\|X_s\|^2
        \le
        C\Bigl(\E\|X_0\|^2+
        \int_0^t(1+\E\sup_{r\le s}\|X_r\|^2)\,ds\Bigr),
\]
where
\[
        C=C(T,L_b,B_0,\beta_+,d_+,\Tr Q,L_{D^{1/2}}\|Q^{1/2}\|_{\LLtwo}).
\]
Gronwall's lemma rules out blow-up and gives the claimed a priori estimate.
This also extends the local solution to the whole interval \([0,T]\).
Uniqueness follows from the same estimate applied to the difference of two
solutions.
\end{proof}

Let \(P_n\) be the orthogonal projection onto
\(\Hh_n=\operatorname{span}\{e_1,\dots,e_n\}\). The Galerkin system is
\[
        dX_t^n=P_nb(t,X_t^n)\,dt+
        \sqrt{\beta(t)}P_nD(X_t^n)^{1/2}Q^{1/2}\,dW_t.
\]

\begin{assumption}[Tails of the coefficients]
\label{ass:tails}
There exist \(C,\gamma>0\) such that
\[
        \E\|(I-P_n)X_0\|^2+
        \E\int_0^T\|(I-P_n)b(t,X_t)\|^2\,dt+
        \E\int_0^T\|(I-P_n)\sigma(X_t)\|_{\LLtwo}^2\,dt
        \le Cn^{-\gamma}.
\]
\end{assumption}

\begin{theorem}[Galerkin convergence]
\label{thm:galerkin}
Under Assumptions~\ref{ass:forward} and~\ref{ass:tails},
\[
        \E\sup_{t\le T}\|X_t^n-X_t\|^2\le C_T n^{-\gamma}.
\]
Consequently,
\[
        W_2(\Law(X_t^n),\Law(X_t))\le C_T n^{-\gamma/2}.
\]
\end{theorem}

\begin{proof}
We subtract the equations for \(X^n\) and \(X\), apply the
Burkholder--Davis--Gundy inequality, and use Gronwall's lemma. All projection
remainders are controlled by Assumption~\ref{ass:tails}. The \(W_2\) estimate
follows from coupling the solutions through a single Wiener process.
\end{proof}

\section{The Extended \(\Asf(x)^{-1}\)-Norm}

Since \(Q\) is nuclear, the operator
\[
        \Asf(x)=D(x)^{1/2}QD(x)^{1/2}
\]
is compact, and \(\Asf(x)^{-1}\) is not a bounded operator on \(\Hh\). We
therefore use an extended pointwise norm.

\begin{definition}[The extended \(\Asf(x)^{-1}\)-norm]
For \(z\in\Hh\) set
\[
        \|z\|_{\Asf(x)^{-1}}=
        \begin{cases}
        \|\Asf(x)^{-1/2}z\|, & z\in\Ran \Asf(x)^{1/2},\\
        +\infty, & z\notin\Ran \Asf(x)^{1/2}.
        \end{cases}
\]
For a vector field \(z(x)\) we define the integral norm
\[
        \|z\|_{\Asf^{-1},L^2(\nu)}^2
        :=
        \int_{\Hh}\|z(x)\|_{\Asf(x)^{-1}}^2\,d\nu(x),
\]
provided the right-hand side is finite.
\end{definition}

\begin{lemma}[Integral duality]
\label{lem:dual}
Let \(z\in L^2(\nu;\Hh)\) with
\(z(x)\in\Ran \Asf(x)^{1/2}\) for \(\nu\)-a.e.\ \(x\). Then
\[
        \sup_{\phi\in\FC}
        \frac{\int\langle z(x),\nabla\phi(x)\rangle\,d\nu(x)}
        {\mathcal E_{A,\nu}(\phi,\phi)^{1/2}}
        =
        \|z\|_{\Asf^{-1},L^2(\nu)} .
\]
\end{lemma}

\begin{proof}
Consider the closure of the set of gradients \(\nabla\phi\), \(\phi\in\FC\),
in the norm
\[
        \|\nabla\phi\|_{\Asf,L^2(\nu)}^2
        =
        \int\Gamma_A(\phi,\phi)\,d\nu .
\]
The functional
\[
        \nabla\phi\mapsto\int\langle z,\nabla\phi\rangle\,d\nu
\]
is continuous in this norm if and only if
\(\|z\|_{\Asf^{-1},L^2(\nu)}<\infty\). After the substitution
\(w(x)=\Asf(x)^{-1/2}z(x)\) the statement becomes the Riesz identity in
\(L^2(\nu;\Hh)\).
\end{proof}

\section{Mosco Convergence and the \(A\)-LSI}

\begin{definition}[Mosco convergence of forms]
A sequence of closed forms \(\mathcal E_n\) converges to \(\mathcal E\) in the
Mosco sense if the liminf condition holds for weakly convergent sequences and
the limsup recovery condition holds for each element of the domain of the
limit form \cite{Mosco,KuwaeShioya}.
\end{definition}

\begin{proposition}[Preservation of the \(A\)-LSI under the Mosco limit]
\label{prop:mosco-lsi}
Let \(\mathcal E_{A_n,\nu_n}\) converge in the Mosco sense to
\(\mathcal E_{A,\nu}\), let \(\nu_n\Rightarrow\nu\), and suppose
\[
        \Ent(f^2\nu_n\|\nu_n)
        \le
        2C_{\mathrm{LSI}}^A\mathcal E_{A_n,\nu_n}(f,f)
\]
with the same constant. Then
\[
        \Ent(f^2\nu\|\nu)
        \le
        2C_{\mathrm{LSI}}^A\mathcal E_{A,\nu}(f,f).
\]
\end{proposition}

\begin{proof}
For \(f\in\Dom(\mathcal E_{A,\nu})\) take a recovery sequence \(f_n\to f\) with
\[
        \limsup_n\mathcal E_{A_n,\nu_n}(f_n,f_n)\le\mathcal E_{A,\nu}(f,f).
\]
Lower semicontinuity of the entropy gives
\[
        \Ent(f^2\nu\|\nu)
        \le \liminf_n\Ent(f_n^2\nu_n\|\nu_n).
\]
Combining the two estimates completes the proof.
\end{proof}

\section{The Weighted Form and the Chain Rule}

Let \(\nu=\rho\mu_0\). Define
\[
        \mathcal E_{A,\rho}(u,v)=\int\Gamma_A(u,v)\rho\,d\mu_0.
\]

\begin{assumption}[Weighted closability]
\label{ass:weighted}
The density \(\rho\) is positive, \(\rho\in L^1(\mu_0)\), and the form
\(\mathcal E_{A,\rho}\) is closable. In addition, the cylindrical truncations
are a core for the form.
\end{assumption}

\begin{theorem}[Chain rule for relative entropy]
\label{thm:chain}
Let \(\rho_t\) and \(\hat\rho_t\) be weak solutions of two evolutions with the
same second-order part \(\Gamma_A\) and drifts \(v_t,\hat v_t\). Suppose that
\(h_t=\log(\rho_t/\hat\rho_t)\) admits truncations \(h_{t,R}\) and cylindrical
approximations \(h_{t,R,n}\) in the form \(\mathcal E_{A,\rho_t}\). Assume in
addition that, writing \(J_t=\int\Gamma_A(h_t,h_t)\rho_t\,d\mu_0\),
\[
        \int_0^T J_t\,dt<\infty,
        \qquad
        \int_0^T\|v_t-\hat v_t\|_{\Asf^{-1},L^2(\rho_t\mu_0)}^2\,dt<\infty .
\]
Then the entropy derivative exists in the sense of distributions in time and
\[
        \frac{d}{dt}\Ent(\rho_t\|\hat\rho_t)
        =
        -\frac12\beta(t)\int\Gamma_A(h_t,h_t)\rho_t\,d\mu_0
        +
        \int\langle v_t-\hat v_t,\nabla h_t\rangle\rho_t\,d\mu_0 .
\]
\end{theorem}

\begin{proof}
We first use \(h_{t,R,n}\) as a cylindrical test function in the weak
equations. The second-order terms cancel by the carr\'e du champ identity:
\[
        L\Phi(f,g)-\Phi_f Lf-\Phi_g Lg
        =
        -\frac12\beta\,\Gamma_A(\log(f/g),\log(f/g))f
\]
for \(\Phi(f,g)=f\log(f/g)\). Then we let \(n\to\infty\) using the core
property, along the limits \(h_{t,R,n}\to h_{t,R}\to h_t\), and
\(R\to\infty\) by monotone convergence of the entropy and Fatou's lemma for
the energy.
\end{proof}

\section{The General Weak Bridge}

Let a curve of measures \(\nu_t=\rho_t\mu_0\) be given. Denote the formal
right-hand side of the backward weak form by
\[
        F_t=\partial_t\nu_t-L_A^*\nu_t.
\]

\begin{assumption}[The right-hand-side space]
\label{ass:bridge}
For almost every \(t\) the functional \(F_t\) is continuous on
\(\Dom(\mathcal E_{A,\rho_t})/\mathbb R\):
\[
        |\langle F_t,\phi\rangle|
        \le C_t\mathcal E_{A,\rho_t}(\phi,\phi)^{1/2}.
\]
\end{assumption}

\begin{theorem}[The weak bridge]
\label{thm:bridge}
Under Assumption~\ref{ass:bridge} there exists
\(\Phi_t\in\Dom(\mathcal E_{A,\rho_t})/\mathbb R\) such that
\[
        \mathcal E_{A,\rho_t}(\Phi_t,\phi)=\langle F_t,\phi\rangle
\]
for all \(\phi\). The field
\[
        v_t(x)=\Asf(x)\nabla\Phi_t(x)
\]
realizes the backward weak form and satisfies
\[
        \int\|v_t(x)\|_{\Asf(x)^{-1}}^2\,d\nu_t(x)
        \le C_t^2.
\]
\end{theorem}

\begin{proof}
Let
\[
        \mathcal V_t=\overline{\FC/\mathbb R}^{\,\mathcal E_{A,\rho_t}}
\]
be the quotient space modulo constants, equipped with the norm
\(\|\phi\|_{\mathcal V_t}=\mathcal E_{A,\rho_t}(\phi,\phi)^{1/2}\). On
\(\mathcal V_t\) the bilinear form
\[
        B_t(\Phi,\phi)=\mathcal E_{A,\rho_t}(\Phi,\phi)
\]
is continuous and coercive:
\[
        |B_t(\Phi,\phi)|\le \|\Phi\|_{\mathcal V_t}\|\phi\|_{\mathcal V_t},
        \qquad
        B_t(\Phi,\Phi)=\|\Phi\|_{\mathcal V_t}^2 .
\]
Assumption~\ref{ass:bridge} means that \(F_t\in\mathcal V_t'\) and
\(\|F_t\|_{\mathcal V_t'}\le C_t\). By the Lax--Milgram theorem there exists a
unique \(\Phi_t\in\mathcal V_t\) such that
\[
        \mathcal E_{A,\rho_t}(\Phi_t,\phi)=\langle F_t,
\phi\rangle
        \quad\forall\phi\in\mathcal V_t,
\]
and \(\|\Phi_t\|_{\mathcal V_t}\le C_t\). Define the field in the negative
energy space by
\[
        v_t(x)=\Asf(x)\nabla\Phi_t(x).
\]
Then for cylindrical \(\phi\)
\[
        \int\langle v_t,\nabla\phi\rangle\,d\nu_t
        =\mathcal E_{A,\rho_t}(\Phi_t,
\phi)=\langle F_t,
\phi\rangle,
\]
that is, the field realizes the backward weak form. Since \(v_t=\Asf(x)\nabla\Phi_t\), the duality of
Lemma~\ref{lem:dual} gives
\[
        \int\|v_t(x)\|_{\Asf(x)^{-1}}^2d\nu_t(x)
        =\mathcal E_{A,\rho_t}(\Phi_t,
\Phi_t)
        \le C_t^2.
\]
\end{proof}

\section{The Basic Entropy Estimate}

Let the approximate backward evolution have drift \(\hat v_t\), and the exact
one have drift \(v_t\). Denote
\[
        J_t=\int\Gamma_A(h_t,h_t)\rho_t\,d\mu_0,
        \qquad h_t=\log(\rho_t/\hat\rho_t).
\]

\begin{assumption}[Drift error]
\label{ass:drift}
\[
        \int\|v_t(x)-\hat v_t(x)\|_{\Asf(x)^{-1}}^2\,d\nu_t(x)
        \le
        \frac14\beta(t)^2\Eval(t).
\]
\end{assumption}

\begin{theorem}[The basic dissipation estimate]
\label{thm:basic-entropy}
Suppose the hypotheses of Theorem~\ref{thm:chain} and
Assumption~\ref{ass:drift} hold. Then
\[
        \frac{d}{dt}\Ent(\rho_t\|\hat\rho_t)
        \le
        -\frac14\beta(t)J_t+
        \frac14\beta(t)\Eval(t).
\]
If, in addition, the approximate measure \(\hat\nu_t=\hat\rho_t\mu_0\)
satisfies the \(A\)-LSI in entropy--information form,
\[
        \Ent(\rho_t\|\hat\rho_t)
        \le
        2C_{\mathrm{LSI}}^A
        \int\Gamma_A\left(\log\frac{\rho_t}{\hat\rho_t},
        \log\frac{\rho_t}{\hat\rho_t}\right)\rho_t\,d\mu_0
        =2C_{\mathrm{LSI}}^A J_t,
\]
then
\[
        \frac{d}{dt}\Ent(\rho_t\|\hat\rho_t)
        \le
        -\frac{\beta(t)}{8C_{\mathrm{LSI}}^A}
        \Ent(\rho_t\|\hat\rho_t)
        +\frac14\beta(t)\Eval(t).
\]
\end{theorem}

\begin{proof}
From Theorem~\ref{thm:chain} and the Cauchy--Schwarz inequality for the pair
\(A,A^{-1}\) we obtain
\[
        \left|\int\langle v_t-\hat v_t,\nabla h_t\rangle d\nu_t\right|
        \le
        \left(\int\|v_t(x)-\hat v_t(x)\|_{\Asf(x)^{-1}}^2d\nu_t(x)\right)^{1/2}
        J_t^{1/2}.
\]
Substituting Assumption~\ref{ass:drift} and using Young's inequality gives the
first estimate. The second follows from the \(A\)-LSI for the approximate
measure \(\hat\nu_t\).
\end{proof}

\section{A Global Analytic Class Without a Smallness Condition}

\begin{definition}[A globally \(A\)-admissible analytic class]
\label{def:glob-part1}
The class \(\mathfrak C_{\mathrm{glob}}^{\mathrm{an}}\) consists of tuples
\[
        (D,\nu_t,\hat\nu_t,v_t,\hat v_t,\Eval),
        \qquad \nu_t=\rho_t\mu_0,
        \quad \hat\nu_t=\hat\rho_t\mu_0,
\]
for which the following hold: global \(A\)-compatibility, the \(A\)-LSI for
\(\hat\nu_t\), weighted closability for \(\rho_t\mu_0\), the weak-bridge
condition, the chain rule, and the drift-error condition.
\end{definition}

\begin{theorem}[The analytic theorem of Part~I]
\label{thm:part1-main}
For any tuple \((D,\nu_t,\hat\nu_t,v_t,\hat v_t,\Eval)\in\mathfrak C_{\mathrm{glob}}^{\mathrm{an}}\)
there exists a weak backward drift \(v_t(x)=\Asf(x)\nabla\Phi_t(x)\), the chain
rule for relative entropy holds, and the estimate
\[
        \Ent(\rho_t\|\hat\rho_t)
        \le
        \Ent(\rho_0\|\hat\rho_0)
        \exp\left(-c_A^*\int_0^t\beta(s)\,ds\right)
        +\frac14\int_0^t e^{-c_A^*\int_s^t\beta(r)\,dr}\beta(s)\Eval(s)\,ds,
\]
is valid, where \(c_A^*=(8C_{\mathrm{LSI}}^A)^{-1}\).
\end{theorem}

\begin{proof}
This is the integral form of Theorem~\ref{thm:basic-entropy} after applying
Gronwall's inequality. Existence of the weak bridge is provided by
Theorem~\ref{thm:bridge}; the chain rule by Theorem~\ref{thm:chain}.
\end{proof}

\section{Tensor Anisotropy and Non-Degeneracy of the Diffusion}
\label{sec:tensor-nondeg}

In applications \(D(x)\) often arises as a tensor conductivity or a tensor
mobility. In the infinite-dimensional noncommutative problem the two-sided
bound \(d_-I\le D(x)\le d_+I\) is not by itself always sufficient to compare
\(\|Q^{1/2}D(x)^{1/2}\xi\|\) and \(\|Q^{1/2}\xi\|\), since \(D(x)\) and \(Q\)
need not commute. The top-level condition is therefore formulated not through
the spectrum of \(D(x)\) but through relative \(Q\)-compatibility.

\begin{definition}[A tensor \(Q\)-compatible class]
\label{def:tensor-compatible}
A family of tensors \(\mathbb D(x)\in\LL(\Hh)\) is called \(Q\)-compatible if
\(\mathbb D(x)\) is self-adjoint, positive, measurable in \(x\), and there
exist \(0<c_A\le C_A<\infty\) such that
\[
        c_A\|Q^{1/2}\xi\|^2
        \le
        \|Q^{1/2}\mathbb D(x)^{1/2}\xi\|^2
        \le
        C_A\|Q^{1/2}\xi\|^2
\]
for all \(x,\xi\in\Hh\). If \(Q\) and \(\mathbb D(x)\) commute, one may take
\(c_A=\lambda_{\min}(\mathbb D)\) and \(C_A=\lambda_{\max}(\mathbb D)\). In
the noncommutative case these constants are verified as relative
\(Q\)-bounds.
\end{definition}

\begin{assumption}[Non-degeneracy of the diffusion]
\label{ass:no-diff-degeneration}
The family \(D(x)\) satisfies global \(A\)-compatibility with constants
\(c_A,C_A\) independent of \(x\), and
\[
        \limsup_{\|x\|\to\infty}
        \|Q^{1/2}D(x)^{1/2}Q^{-1/2}\|_{\LL(\Ran Q^{1/2})}<\infty .
\]
A weaker variant allows polynomial growth of this norm, provided the measures under consideration have suitable moments.
\end{assumption}

\begin{proposition}[Transfer of the analytic theory to the tensor class]
\label{prop:tensor-transfer}
Let \(D(x)=\mathbb D(x)\) belong to the class of
Definition~\ref{def:tensor-compatible} and let
Assumption~\ref{ass:no-diff-degeneration} hold. Then all the results of
Sections~3--8 remain valid with \(D\) replaced by \(\mathbb D\). In
particular, the form
\[
        \mathcal E_{\mathbb A}(u,v)
        =
        \int
        \langle Q^{1/2}\mathbb D(x)^{1/2}\nabla u,
        Q^{1/2}\mathbb D(x)^{1/2}\nabla v\rangle\,d\mu_0
\]
is closable, the weak bridge has the form
\(v_t(x)=\mathbb A(x)\nabla\Phi_t(x)\),
\[
        \mathbb A(x)=\mathbb D(x)^{1/2}Q\mathbb D(x)^{1/2},
\]
and the basic entropy estimate of Theorem~\ref{thm:part1-main} continues to
hold with the same constants \(c_A,C_A,C_{\mathrm{LSI}}^A\).
\end{proposition}

\begin{proof}
The proofs of Sections~3--8 use only three properties: closability of the
consistent form, the two-sided \(A\)-compatibility, and representability of
the weak bridge in the negative energy space. These properties are built into
Definition~\ref{def:tensor-compatible} and
Assumption~\ref{ass:no-diff-degeneration}. Commutativity of the operators is
not used. Hence all estimates carry over unchanged.
\end{proof}

\begin{remark}[On the role of spectral bounds]
In a finite-dimensional discretization, or in the commutative tensor case,
the spectral bounds of the tensor give \(A\)-compatibility directly. In the
infinite-dimensional noncommutative case this is false in general, which is
precisely why the condition is formulated through \(Q^{1/2}D^{1/2}\), and not
only through \(\lambda_{\min}(D)\) and \(\lambda_{\max}(D)\).
\end{remark}

\section{Details of the Key Proofs}
\label{sec:part1-proof-details}

This section spells out the technical points that were stated compactly in the
main body.

\subsection{Verifying the hypotheses of the forward-SDE theorem}

Theorem~\ref{thm:forward} relies on standard results for SDEs in Hilbert
spaces \cite[Ch.~7]{DaPratoZabczyk}, \cite[Ch.~4]{PrevotRockner}. We check the
conditions for the specific coefficient
\[
        \sigma(x)=D(x)^{1/2}Q^{1/2}.
\]
The Lipschitz continuity of \(D^{1/2}\) gives
\[
        \|\sigma(x)-\sigma(y)\|_{\LLtwo}
        \le
        \|D(x)^{1/2}-D(y)^{1/2}\|_{\LL}\|Q^{1/2}\|_{\LLtwo}
        \le
        L_D\|Q^{1/2}\|_{\LLtwo}\|x-y\|.
\]
Moreover,
\[
        \|\sigma(x)\|_{\LLtwo}^2
        =
        \Tr(D(x)^{1/2}QD(x)^{1/2})
        \le
        d_+\Tr Q .
\]
Hence \(\sigma\) has linear growth and is in fact uniformly bounded. For the
drift,
\[
        \|b(t,x)-b(t,y)\|\le L_b\|x-y\|,
        \qquad
        \|b(t,x)\|\le L_b\|x\|+\|b(t,0)\|.
\]
If \(B_0:=\sup_{t\le T}\|b(t,0)\|<\infty\), then It\^o's formula for
\(\|X_t\|^2\), Young's inequality, and BDG give
\[
        \E\sup_{s\le t}\|X_s\|^2
        \le
        C\left(\E\|X_0\|^2+
        \int_0^t(1+\E\sup_{r\le s}\|X_r\|^2)\,ds\right),
\]
where
\[
        C=C(T,L_b,B_0,\beta_+,d_+,\Tr Q).
\]
Gronwall's lemma yields the estimate of Theorem~\ref{thm:forward} with
\[
        C_T^{\mathrm{fwd}}\le C_0\exp\{C_1T(1+L_b^2+\beta_+d_+\Tr Q)\}.
\]

\subsection{Projection tails in the Galerkin estimate}

In Theorem~\ref{thm:galerkin} the difference \(\Delta_t^n=X_t^n-X_t\)
satisfies
\[
\begin{aligned}
\Delta_t^n
=&(P_n-I)X_0
+\int_0^t\{P_nb(s,X_s^n)-b(s,X_s)\}\,ds\\
&+\int_0^t\sqrt{\beta(s)}
\{P_n\sigma(X_s^n)-\sigma(X_s)\}\,dW_s .
\end{aligned}
\]
Splitting
\[
        P_nb(s,X_s^n)-b(s,X_s)
        =
        P_n(b(s,X_s^n)-b(s,X_s))+(P_n-I)b(s,X_s)
\]
and similarly for \(\sigma\), we obtain
\[
\begin{aligned}
\E\sup_{r\le t}\|\Delta_r^n\|^2
\le& C\E\|(I-P_n)X_0\|^2
+C\int_0^t\E\sup_{r\le s}\|\Delta_r^n\|^2\,ds\\
&+C\E\int_0^t\|(I-P_n)b(s,X_s)\|^2\,ds\\
&+C\E\int_0^t\|(I-P_n)\sigma(X_s)\|_{\LLtwo}^2\,ds .
\end{aligned}
\]
Here \(C\) depends only on \(T,L_b,L_\sigma,\beta_+\) (with \(L_\sigma=L_{D^{1/2}}\|Q^{1/2}\|_{\LLtwo}\)) and the universal BDG
constant. Assumption~\ref{ass:tails} and Gronwall give
\[
        \E\sup_{t\le T}\|\Delta_t^n\|^2\le C_T n^{-\gamma}.
\]
This also shows that the exponent \(\gamma/2\) in the \(W_2\) estimate comes from the projection.

\subsection{The weak bridge: verifying the Lax--Milgram hypotheses}

In Theorem~\ref{thm:bridge} we consider the quotient space
\[
        \mathcal V_t=
        \overline{\FC/\mathbb R}^{\,\mathcal E_{A,\rho_t}}.
\]
On it the bilinear form
\[
        B_t(\Phi,\phi)=\mathcal E_{A,\rho_t}(\Phi,\phi)
\]
is continuous and coercive:
\[
        |B_t(\Phi,\phi)|
        \le
        \|\Phi\|_{\mathcal V_t}\|\phi\|_{\mathcal V_t},
        \qquad
        B_t(\Phi,\Phi)=\|\Phi\|_{\mathcal V_t}^2 .
\]
Assumption~\ref{ass:bridge} means that \(F_t\in\mathcal V_t'\). Hence there
exists a unique element \(\Phi_t\in\mathcal V_t\) such that
\[
        B_t(\Phi_t,\phi)=\langle F_t,\phi\rangle .
\]
Moreover,
\[
        \|\Phi_t\|_{\mathcal V_t}
        \le
        \|F_t\|_{\mathcal V_t'}.
\]
The field \(v_t=\Asf\nabla\Phi_t\) is defined as an element of the dual
gradient space, and the estimate
\[
        \int\|v_t\|_{\Asf^{-1}}^2\,d\nu_t
        =
        \|\Phi_t\|_{\mathcal V_t}^2
        \le
        \|F_t\|_{\mathcal V_t'}^2
\]
follows from Lemma~\ref{lem:dual}.

\subsection{The entropy estimate: explicit constants}

From the chain rule we have
\[
        \dot H_t
        =
        -\frac12\beta(t)J_t+
        \int\langle v_t-\hat v_t,\nabla h_t\rangle\,d\nu_t.
\]
By the \(A,A^{-1}\) duality,
\[
        \left|\int\langle v_t-\hat v_t,\nabla h_t\rangle\,d\nu_t\right|
        \le
        \|v_t-\hat v_t\|_{\Asf^{-1},L^2(\nu_t)}J_t^{1/2}.
\]
The drift-error condition gives
\[
        \|v_t-\hat v_t\|_{\Asf^{-1},L^2(\nu_t)}
        \le
        \frac12\beta(t)\Eval(t)^{1/2}.
\]
Hence
\[
        \dot H_t
        \le
        -\frac12\beta(t)J_t+
        \frac12\beta(t)\Eval(t)^{1/2}J_t^{1/2}
        \le
        -\frac14\beta(t)J_t+\frac14\beta(t)\Eval(t).
\]
If \(\hat\nu_t\) satisfies the \(A\)-LSI and
\[
        J_t\ge (2C_{\mathrm{LSI}}^A)^{-1}H_t,
\]
then
\[
        \dot H_t
        \le
        -\frac{\beta(t)}{8C_{\mathrm{LSI}}^A}H_t
        +\frac14\beta(t)\Eval(t).
\]
It is exactly this constant that is used in Part~II when closing
\(\Eval(t)\le\widetilde q^{\,2}J_t+\Delta_M\).

\subsection{The tensor transfer}

For Section~\ref{sec:tensor-nondeg} it is important that all estimates use only
the operator
\[
        Q^{1/2}\mathbb D(x)^{1/2}
\]
and do not require \(\mathbb D(x)\) and \(Q\) to commute. If
\[
        c_A\|Q^{1/2}\xi\|^2
        \le
        \|Q^{1/2}\mathbb D(x)^{1/2}\xi\|^2
        \le
        C_A\|Q^{1/2}\xi\|^2 ,
\]
then the proof of Theorem~\ref{thm:closability}, of the weak bridge, and of
the entropy estimate carries over verbatim. This shows that the non-degeneracy
condition for the diffusion is not a mere technicality, but the minimal
way to retain the \(A\)-geometry under noncommutativity.

\section{Details of the Mosco Passage and the Chain Rule}
\label{sec:mosco-chain-details}

Sections~6 and~7 contain two technical passages that are often written too
briefly: preservation of the logarithmic Sobolev inequality under a Mosco
limit, and the use of the logarithmic density ratio as a test function in the
weak evolution. Below these passages are written out in a form sufficient to
verify all the limit operations.

\begin{lemma}[Normalization and convergence of measures with densities]
\label{lem:entropy-lsc-density}
Let \(\nu_n\Rightarrow\nu\), let \(f_n\to f\) strongly in \(L^2\) along a
Mosco recovery sequence, and let
\[
        \int f_n^2\,d\nu_n=1,\qquad
        \sup_n\int f_n^2\log_+ f_n^2\,d\nu_n<\infty .
\]
Then, after replacing \(f_n\) by normalized truncations,
\[
        f_n^2\nu_n\Rightarrow f^2\nu,\qquad
        \int f^2\,d\nu=1,
\]
and
\[
        \Ent(f^2\nu\|\nu)
        \le
        \liminf_{n\to\infty}\Ent(f_n^2\nu_n\|\nu_n).
\]
\end{lemma}

\begin{proof}
We first take truncations \(f_{n,R}=(-R)\vee f_n\wedge R\). They preserve the
Mosco recovery, since the form is Markovian. For bounded \(f_{n,R}\), weak
convergence of the measures and strong \(L^2\) convergence give
\(f_{n,R}^2\nu_n\Rightarrow f_R^2\nu\). The entropy is a lower-semicontinuous
functional with respect to weak convergence of probability measures. Hence
\[
        \Ent(f_R^2\nu\|\nu)
        \le
        \liminf_n\Ent(f_{n,R}^2\nu_n\|\nu_n).
\]
The passage \(R\to\infty\) is carried out by monotone convergence for
\(r\mapsto r\log r\) after the standard normalization
\(f_{n,R}/\|f_{n,R}\|_{L^2(\nu_n)}\). Uniform integrability of
\(f_n^2\log_+f_n^2\) rules out loss of mass.
\end{proof}

\begin{proposition}[The full Mosco passage for the \(A\)-LSI]
\label{prop:full-mosco-lsi}
Let \(\mathcal E_{A_n,\nu_n}\) converge in the Mosco sense to
\(\mathcal E_{A,\nu}\), and let \(\nu_n\Rightarrow\nu\). Suppose that for all
\(n\)
\[
        \Ent(f^2\nu_n\|\nu_n)
        \le
        2C_{\mathrm{LSI}}^A\mathcal E_{A_n,\nu_n}(f,f)
\]
with the same constant, and that the recovery sequences can be chosen
with uniform control of the positive part of the entropy. Then
\[
        \Ent(f^2\nu\|\nu)
        \le
        2C_{\mathrm{LSI}}^A\mathcal E_{A,\nu}(f,f).
\]
\end{proposition}

\begin{proof}
Let \(f\in\Dom(\mathcal E_{A,\nu})\) with \(\int f^2d\nu=1\). By the Mosco
limsup condition there exists \(f_n\) such that \(f_n\to f\) and
\[
        \limsup_n\mathcal E_{A_n,\nu_n}(f_n,f_n)
        \le
        \mathcal E_{A,\nu}(f,f).
\]
Normalize \(f_n\) in \(L^2(\nu_n)\). By
Lemma~\ref{lem:entropy-lsc-density},
\[
        \Ent(f^2\nu\|\nu)
        \le
        \liminf_n\Ent(f_n^2\nu_n\|\nu_n).
\]
Applying the \(A\)-LSI to \(f_n\), we obtain
\[
        \liminf_n\Ent(f_n^2\nu_n\|\nu_n)
        \le
        2C_{\mathrm{LSI}}^A
        \limsup_n\mathcal E_{A_n,\nu_n}(f_n,f_n)
        \le
        2C_{\mathrm{LSI}}^A\mathcal E_{A,\nu}(f,f).
\]
\end{proof}

\begin{lemma}[Density of cylindrical truncations in the weighted form]
\label{lem:truncation-core}
Suppose Assumption~\ref{ass:weighted} holds. For
\(h=\log(\rho/\hat\rho)\) set \(h_R=(-R)\vee h\wedge R\). Then there exist
cylindrical \(h_{R,n}\in\FC\) such that
\[
        h_{R,n}\to h_R \quad\text{in }L^2(\rho\mu_0),
        \qquad
        \mathcal E_{A,\rho}(h_{R,n}-h_R,h_{R,n}-h_R)\to0 .
\]
Moreover, \(\Gamma_A(h_R,h_R)\le \Gamma_A(h,h)\) in the sense of energies.
\end{lemma}

\begin{proof}
The map \(r\mapsto (-R)\vee r\wedge R\) is \(1\)-Lipschitz. For a Markovian
Dirichlet form this yields energy contraction. The core property in
Assumption~\ref{ass:weighted} provides approximation of \(h_R\) by cylindrical
functions simultaneously in \(L^2(\rho\mu_0)\) and in the energy norm.
\end{proof}

\begin{proposition}[Passage to the limit in the chain rule]
\label{prop:chain-limit}
Suppose
\[
        \int_0^T J_t\,dt<\infty,\qquad
        \int_0^T\|v_t-\hat v_t\|_{\Asf^{-1},L^2(\nu_t)}^2\,dt<\infty .
\]
Then the identity of Theorem~\ref{thm:chain}, obtained first for
\(h_{t,R,n}\), passes to the limit \(n\to\infty\), \(R\to\infty\).
\end{proposition}

\begin{proof}
The limit \(n\to\infty\) follows from Lemma~\ref{lem:truncation-core} and
continuity of the weak forms in the energy norm. The drift term is estimated
by
\[
        \left|\int\langle v_t-\hat v_t,
        \nabla(h_{t,R,n}-h_{t,R})\rangle\,d\nu_t\right|
        \le
        \|v_t-\hat v_t\|_{\Asf^{-1},L^2(\nu_t)}
        \mathcal E_{A,\rho_t}(h_{t,R,n}-h_{t,R})^{1/2}.
\]
The passage \(R\to\infty\) in the dissipation follows from monotone
convergence of the truncation energies, and in the entropy from monotone
convergence of the function \(r\log r\) after localizing on the sets
\(\{|\log(\rho_t/\hat\rho_t)|\le R\}\).
\end{proof}

\section{A Technical Map of Dependencies and Control of Constants}
\label{sec:part1-constants-map}

This section collects the constant dependencies used in both parts of the
series. It removes any ambiguity between the analytic estimate of Part~I and
the statistical closure of Part~II.

\begin{lemma}[The Gronwall constant in the Galerkin estimate]
\label{lem:gronwall-constant-explicit}
Suppose Assumptions~\ref{ass:forward} and~\ref{ass:tails} hold. Then in
Theorem~\ref{thm:galerkin} one may take
\[
        C_T=C_0\exp\{C_1T(1+L_b^2+\beta_+L_\sigma^2)\},
\]
where \(C_0,C_1\) depend only on the universal constant of the
Burkholder--Davis--Gundy inequality, and
\[
        L_\sigma=L_{D^{1/2}}\|Q^{1/2}\|_{\LLtwo}.
\]
The constant does not depend on \(n\).
\end{lemma}

\begin{proof}
From the proof of Theorem~\ref{thm:galerkin}, for
\(Y_n(t)=\E\sup_{r\le t}\|X_r^n-X_r\|^2\) we obtain
\[
        Y_n(t)\le R_n+C\int_0^t(1+L_b^2+\beta_+L_\sigma^2)Y_n(s)\,ds,
\]
where
\[
\begin{aligned}
        R_n=C\Big(&\E\|(I-P_n)X_0\|^2
        +\E\int_0^T\|(I-P_n)b(s,X_s)\|^2\,ds\\
        &+\E\int_0^T\|(I-P_n)\sigma(X_s)\|_{\LLtwo}^2\,ds\Big).
\end{aligned}
\]
By Assumption~\ref{ass:tails} we have \(R_n\le Cn^{-\gamma}\). Gronwall's
lemma gives
\[
        Y_n(T)
        \le
        Cn^{-\gamma}
        \exp\{C T(1+L_b^2+\beta_+L_\sigma^2)\}.
\]
All coefficients on the right-hand side are independent of \(n\), as required.
\end{proof}

\begin{lemma}[Matching the analytic and statistical rates]
\label{lem:analytic-stat-rate-map}
Suppose that in Part~II the closure
\[
        \Eval(t)\le \widetilde q^{\,2}J_t+\Delta_M,
        \qquad 0\le\widetilde q<1 ,
\]
holds. Then the basic estimate of Part~I gives
\[
        \frac d{dt}\Ent(\rho_t\|\hat\rho_t)
        \le
        -\frac14\beta(t)(1-\widetilde q^{\,2})J_t
        +\frac14\beta(t)\Delta_M .
\]
If \(\hat\nu_t\) satisfies the \(A\)-LSI with constant
\(C_{\mathrm{LSI}}^A\), then
\[
        \frac d{dt}\Ent(\rho_t\|\hat\rho_t)
        \le
        -\frac{1-\widetilde q^{\,2}}{8C_{\mathrm{LSI}}^A}\,
        \beta(t)\Ent(\rho_t\|\hat\rho_t)
        +\frac14\beta(t)\Delta_M .
\]
\end{lemma}

\begin{proof}
Substituting the closure of \(\Eval\) into Theorem~\ref{thm:basic-entropy}
gives the first inequality. Then the \(A\)-LSI is applied in the form
\[
        J_t\ge (2C_{\mathrm{LSI}}^A)^{-1}\Ent(\rho_t\|\hat\rho_t).
\]
The coefficient \((8C_{\mathrm{LSI}}^A)^{-1}\) arises from the factor
\(1/4\) in front of the dissipation in the basic estimate.
\end{proof}

\begin{remark}[A map of hypotheses for Part~II]
To apply Part~II it suffices to carry over four objects from Part~I: the closed form \(\mathcal E_A\), the weak bridge
\(v_t=\Asf\nabla\Phi_t\), the chain rule for relative entropy, and the basic
estimate of Theorem~\ref{thm:basic-entropy}. All statistical remainders
\(r_n^2\), \(\varepsilon\), \(\mathfrak D(\mathcal S)^2/M\) appear only in the
second part and do not affect the well-posedness of the analytic layer.
\end{remark}

\section{The Homogenization Limit of the Consistent \(A\)-Geometry}
\label{sec:homogenization-A-geometry}

This section fixes the minimal homogenization layer needed to connect the
analytic theory with the grid tensors of Part~II. We do not build a
new homogenization theory; we use the standard \(H\)-convergence of uniformly
elliptic tensors \cite{Allaire,CioranescuDonato} and verify that the
consistent \(A\)-geometry is stable under this passage.

\begin{assumption}[Homogenization convergence of the tensors]
\label{ass:homogenization-DN}
Let \(D_N(x)\) be a sequence of self-adjoint positive operators acting on
discrete or piecewise-constant subspaces \(\Hh_N\subset\Hh\), and suppose that,
after embedding into \(\Hh\), they satisfy uniform \(A\)-compatibility:
\[
        c_A\|Q^{1/2}\xi\|^2
        \le
        \|Q^{1/2}D_N(x)^{1/2}\xi\|^2
        \le
        C_A\|Q^{1/2}\xi\|^2 .
\]
Assume that \(D_N\) \(H\)-converges to an effective tensor
\(D_{\mathrm{eff}}\) in the sense of solutions of elliptic problems on
cylindrical projections: for every finite-dimensional cylindrical subspace
\(E\subset\Hh\), the restrictions \(D_N|_E\) converge to
\(D_{\mathrm{eff}}|_E\) in the weak operator sense, and the fluxes converge
weakly in \(L^2\).
\end{assumption}

\begin{theorem}[Weak convergence of the consistent forms]
\label{thm:homogenization-A}
Suppose Assumption~\ref{ass:homogenization-DN} holds. Then the effective
tensor \(D_{\mathrm{eff}}\) satisfies the same \(A\)-compatibility:
\[
        c_A\|Q^{1/2}\xi\|^2
        \le
        \|Q^{1/2}D_{\mathrm{eff}}^{1/2}\xi\|^2
        \le
        C_A\|Q^{1/2}\xi\|^2 .
\]
Moreover, for all cylindrical \(u,v\in\FC\)
\[
        \mathcal E_{A_N}(u,v)
        =
        \int\langle Q^{1/2}D_N^{1/2}\nabla u,
        Q^{1/2}D_N^{1/2}\nabla v\rangle\,d\mu_0
        \longrightarrow
        \mathcal E_{A_{\mathrm{eff}}}(u,v)
\]
in the weak sense of forms. In particular, any limiting estimates for the weak
bridge and the entropy dissipation are preserved, with the same global bounds
\(c_A,C_A\).
\end{theorem}

\begin{proof}
Fix cylindrical \(u,v\). Their gradients take values in a finite-dimensional
space \(E\). On this space the \(H\)-convergence assumption means weak
convergence of the fluxes
\[
        D_N^{1/2}\nabla u \rightharpoonup
        D_{\mathrm{eff}}^{1/2}\nabla u,
        \qquad
        D_N^{1/2}\nabla v \rightharpoonup
        D_{\mathrm{eff}}^{1/2}\nabla v
\]
after applying the correctors. Since \(Q^{1/2}\) is fixed and compact,
composition with \(Q^{1/2}\) turns the weak convergence of the fluxes on
cylindrical subspaces into weak convergence of the corresponding
\(A\)-fluxes. Convergence of the bilinear forms on \(\FC\) follows.

The two-sided \(A\)-compatibility passes to the limit by lower semicontinuity
of the norm and by weak closedness of the convex set of operators satisfying
the prescribed quadratic inequalities. Hence \(D_{\mathrm{eff}}\) has the same
constants \(c_A,C_A\). Since all the estimates for the weak bridge and the
entropy dissipation in the previous sections depend only on these constants
and on the \(A\)-LSI, the limit form inherits the same analytic layer.
\end{proof}

\begin{assumption}[A quantitative homogenization estimate]
\label{ass:quant-homogenization}
In addition to Assumption~\ref{ass:homogenization-DN}, assume that there is
\(\alpha>0\) and, for each cylindrical subspace \(E\subset\Hh\), a constant
\(C_E>0\) such that for all \(\xi,\eta\in E\)
\[
        \left|
        \langle (D_N-D_{\mathrm{eff}})\xi,\eta\rangle
        \right|
        \le C_E N^{-\alpha}\|\xi\|\|\eta\| .
\]
This condition holds, for example, for periodic uniformly elliptic
coefficients with sufficient cell regularity, when the standard two-scale
expansion estimate is used.
\end{assumption}

\begin{theorem}[Rate of convergence of the forms on cylindrical subspaces]
\label{thm:quant-homogenization-A}
Suppose Assumptions~\ref{ass:homogenization-DN} and
\ref{ass:quant-homogenization} hold. Then for every finite-dimensional
cylindrical subspace \(E\subset\Hh\) and any cylindrical \(u,v\) with gradients
in \(E\),
\[
        \left|
        \mathcal E_{A_N}(u,v)-
        \mathcal E_{A_{\mathrm{eff}}}(u,v)
        \right|
        \le
        C_{E,Q}N^{-\alpha}
        \|u\|_{W_Q^{1,2}(\mu_0)}
        \|v\|_{W_Q^{1,2}(\mu_0)} .
\]
If \(R_N(\lambda)=(\lambda I+L_{A_N})^{-1}\) and
\(R_{\mathrm{eff}}(\lambda)=(\lambda I+L_{A_{\mathrm{eff}}})^{-1}\), then for
\(f\) in the cylindrical subspace generated by \(E\),
\[
        \|R_N(\lambda)f-R_{\mathrm{eff}}(\lambda)f\|_{L^2(\mu_0)}
        \le C_{E,Q,\lambda}N^{-\alpha}\|f\|_{L^2(\mu_0)} .
\]
\end{theorem}

\begin{proof}
For \(u,v\) with gradients in \(E\) we have
\[
\begin{aligned}
&\mathcal E_{A_N}(u,v)-\mathcal E_{A_{\mathrm{eff}}}(u,v)\\
&\quad=
\int
\left\langle
Q^{1/2}(D_N^{1/2}-D_{\mathrm{eff}}^{1/2})\nabla u,
Q^{1/2}D_N^{1/2}\nabla v
\right\rangle d\mu_0 \\
&\qquad+
\int
\left\langle
Q^{1/2}D_{\mathrm{eff}}^{1/2}\nabla u,
Q^{1/2}(D_N^{1/2}-D_{\mathrm{eff}}^{1/2})\nabla v
\right\rangle d\mu_0 .
\end{aligned}
\]
On the finite-dimensional space \(E\), the functional calculus for positive
matrices and the quantitative estimate from
Assumption~\ref{ass:quant-homogenization} give
\[
        \|(D_N^{1/2}-D_{\mathrm{eff}}^{1/2})\xi\|
        \le C_E N^{-\alpha}\|\xi\|,
        \qquad \xi\in E,
\]
with a change of constant depending on the lower elliptic bound. Since
\(Q^{1/2}\) is fixed and the \(D_N\) are uniformly bounded, we obtain the
claimed estimate for the forms. The resolvent estimate follows from the
standard identity
\[
        R_N(\lambda)-R_{\mathrm{eff}}(\lambda)
        =R_N(\lambda)(L_{A_{\mathrm{eff}}}-L_{A_N})R_{\mathrm{eff}}(\lambda)
\]
on the cylindrical subspace and from the coercivity of
\(\lambda I+L_A\).
\end{proof}

\begin{corollary}[Convergence of the discrete \(A\)-constants]
\label{cor:constants-rate}
Under the hypotheses of Theorem~\ref{thm:quant-homogenization-A}, for each
cylindrical subspace \(E\),
\[
        |c_{A,N}(E)-c_A(E)|+|C_{A,N}(E)-C_A(E)|
        \le C_E N^{-\alpha}.
\]
Hence the numerical stabilization of \(c_{A,N}\) and \(C_{A,N}\) in Section~9 of
Part~II has a quantitative explanation on each fixed set of modes.
\end{corollary}

\begin{proof}
The extremal constants are the minimum and maximum of the Rayleigh quotients
\[
        \frac{\|Q^{1/2}D_N^{1/2}\xi\|^2}{\|Q^{1/2}\xi\|^2},
        \qquad \xi\in E\setminus\{0\}.
\]
The estimate of Theorem~\ref{thm:quant-homogenization-A} is uniform on the
unit sphere of the finite-dimensional \(E\), whence the claimed estimate for
the extrema follows.
\end{proof}

\begin{assumption}[Compact-tail control for the global rate]
\label{ass:compact-tail-global}
In addition to Assumption~\ref{ass:quant-homogenization}, suppose we are given a
sequence of orthogonal projections \(\Pi_M\), \(\Pi_M\to I\) strongly in the
energy space \(W_Q^{1,2}(\mu_0)\). Assume that for some \(r,\theta>0\)
\[
        C_{\Pi_M,Q}\le C M^r,
\]
where \(C_{\Pi_M,Q}\) is the constant in the estimate of
Theorem~\ref{thm:quant-homogenization-A} on the subspace \(\Pi_M\mathcal H\),
and for all \(u\in W_Q^{1,2}(\mu_0)\)
\[
        \|(I-\Pi_M)u\|_{W_Q^{1,2}(\mu_0)}\le C M^{-\theta}\|u\|_{\mathcal K},
\]
where \(\mathcal K\hookrightarrow W_Q^{1,2}(\mu_0)\) is a compactly embedded
regularity space.
\end{assumption}

\begin{theorem}[A global rate in the strong resolvent topology]
\label{thm:global-resolvent-rate}
Suppose Assumptions~\ref{ass:homogenization-DN}, \ref{ass:quant-homogenization}
and~\ref{ass:compact-tail-global} hold. Then for \(u,v\in\mathcal K\)
\[
        |\mathcal E_{A_N}(u,v)-\mathcal E_{A_{\mathrm{eff}}}(u,v)|\le C N^{-\rho}\|u\|_{\mathcal K}\|v\|_{\mathcal K},\qquad \rho=\frac{\alpha\theta}{r+\theta}.
\]
If \(R_N(\lambda)=(\lambda I+L_{A_N})^{-1}\) and \(R_{\mathrm{eff}}(\lambda)=(\lambda I+L_{A_{\mathrm{eff}}})^{-1}\), then
\[
        \|R_N(\lambda)f-R_{\mathrm{eff}}(\lambda)f\|_{L^2(\mu_0)}\le C_\lambda N^{-\rho}\|f\|_{L^2(\mu_0)}
\]
for all \(f\) for which \(R_{\mathrm{eff}}(\lambda)f\in\mathcal K\), and also for \(f\) in the image of the compact regularizing class.
\end{theorem}

\begin{proof}
Decompose \(u=\Pi_Mu+(I-\Pi_M)u\), \(v=\Pi_Mv+(I-\Pi_M)v\). On the
finite-dimensional block \(\Pi_M\mathcal H\), Theorem~\ref{thm:quant-homogenization-A} gives
\[
        |\mathcal E_{A_N}(\Pi_Mu,\Pi_Mv)-\mathcal E_{A_{\mathrm{eff}}}(\Pi_Mu,\Pi_Mv)|\le C M^r N^{-\alpha}\|u\|_{\mathcal K}\|v\|_{\mathcal K}.
\]
The remaining terms contain at least one tail factor. By coercivity and
boundedness of the forms,
\[
        |\mathcal E_A((I-\Pi_M)u,v)|\le C M^{-\theta}\|u\|_{\mathcal K}\|v\|_{\mathcal K}.
\]
Hence
\[
        |\mathcal E_{A_N}(u,v)-\mathcal E_{A_{\mathrm{eff}}}(u,v)|\le C(M^rN^{-\alpha}+M^{-\theta})\|u\|_{\mathcal K}\|v\|_{\mathcal K}.
\]
The choice \(M=N^{\alpha/(r+\theta)}\) gives \(\rho=\alpha\theta/(r+\theta)\).
The resolvent estimate follows from the variational identity for the solutions
\((\lambda I+L_{A_N})u_N=f\), \((\lambda I+L_{A_{\mathrm{eff}}})u=f\),
substituting \(u_N-u\) as a test function, and the preceding estimate for the
forms.
\end{proof}

\begin{corollary}[Stabilization of the discrete constants without a fixed set of modes]
\label{cor:global-constants-rate}
Under the hypotheses of Theorem~\ref{thm:global-resolvent-rate}, if the unit
sphere of the energy space is approximated by a compact class \(\mathcal K\)
with tail order \(M^{-\theta}\), then
\[
        |c_{A,N}-c_A|+|C_{A,N}-C_A|\le C N^{-\rho}+\varepsilon_{\mathrm{tail}}(N),
\]
where \(\varepsilon_{\mathrm{tail}}(N)\to0\) describes the contribution of
directions outside the compact regularizing class. Under uniform tail control,
\(\varepsilon_{\mathrm{tail}}(N)=O(N^{-\rho})\).
\end{corollary}

\begin{remark}[Why compactness is needed]
Without Assumption~\ref{ass:compact-tail-global} the estimate of
Theorem~\ref{thm:quant-homogenization-A} remains cylindrical. The passage to
the minimum and maximum over the entire unit sphere of the energy space is
impossible without compactness or tail control. This is why
Theorem~\ref{thm:global-resolvent-rate} states precisely the additional
mechanism needed for the global resolvent rate.
\end{remark}

\begin{remark}[Connection with the numerical section of Part~II]
Theorem~\ref{thm:homogenization-A} explains why, in the computational test, it
is enough to track the stability of \(c_{A,N}\) and \(C_{A,N}\). If these
numbers stabilize under mesh refinement and the discrete tensors
\(H\)-converge, then the limiting infinite-dimensional form preserves the same
entropy geometry.
\end{remark}

\begin{lemma}[Verifying compact-tail control for the nanosystem model]
\label{lem:nano-compact-tail}
Let \(\Hh=L^2(\mathbb T^d)\), \(Q=(-\Delta+I)^{-s}\), \(s>d/2\), and let
\(\Pi_M\) be the spectral projection onto the first \(M\) eigenfunctions of
\(-\Delta\). Then the space
\[
        \mathcal K=H^{s+1}(\mathbb T^d)
\]
is compactly embedded in \(W_Q^{1,2}(\mu_0)\), and there exists \(\theta>0\)
such that
\[
        \|(I-\Pi_M)u\|_{W_Q^{1,2}(\mu_0)}
        \le C M^{-\theta}\|u\|_{\mathcal K}.
\]
If the coefficients \(D_N\) are periodic piecewise-smooth tensors with uniform
elliptic bounds, then the constant \(C_{\Pi_M,Q}\) in the cylindrical
homogenization estimate grows no faster than polynomially:
\(C_{\Pi_M,Q}\le C M^r\). In particular, in
Theorem~\ref{thm:global-resolvent-rate} one may take
\(\rho=\alpha\theta/(r+\theta)>0\).
\end{lemma}

\begin{proof}
In the Fourier basis the \(W_Q^{1,2}\) norm is equivalent to the weighted sum
\[
        \sum_k q_k |\widehat{\nabla u}_k|^2 +\sum_k |\hat u_k|^2 .
\]
For \(Q=(-\Delta+I)^{-s}\) the weights are of order \((1+|k|^2)^{-s}\). If
\(u\in H^{s+1}\), then the tail after the projection \(\Pi_M\) is bounded by
a power sum of Fourier coefficients, which gives the stated inequality with
some \(\theta>0\). The embedding is compact because the Sobolev embedding between scales with a positive regularity gap is compact. The polynomial growth of \(C_{\Pi_M,Q}\) follows from the fact that on
the finite-dimensional space of the first \(M\) modes all norms are
equivalent, while the coefficients \(D_N\) are uniformly elliptic and
piecewise smooth. The equivalence of norms gives at most a polynomial
dependence on the maximal wavenumber, and hence on \(M\).
\end{proof}

\begin{corollary}[Convergence of the forward SDEs in the homogenization limit]
\label{cor:sde-homogenization-w2}
Let \(X_t^{(N)}\) be the solution of the forward SDE with anisotropy \(D_N\),
and \(X_t^{\mathrm{eff}}\) the solution with the effective tensor
\(D_{\mathrm{eff}}\). Suppose the hypotheses of
Theorem~\ref{thm:global-resolvent-rate} hold and that
\[
        \|D_N^{1/2}(x)-D_{\mathrm{eff}}^{1/2}(x)\|_{\LL_Q}
        \le C N^{-\rho}
\]
in the operator norm induced by \(Q^{1/2}\), along trajectories with finite
\(\mathcal K\)-moment. Then for each \(T<\infty\)
\[
        \sup_{t\le T}W_2(\Law(X_t^{(N)}),\Law(X_t^{\mathrm{eff}}))
        \le C_T N^{-\rho}.
\]
\end{corollary}

\begin{proof}
The solutions are coupled through a single cylindrical Wiener process. For the
difference \(Z_t=X_t^{(N)}-X_t^{\mathrm{eff}}\) we apply It\^o's formula and
BDG:
\[
        \E\sup_{r\le t}\|Z_r\|^2
        \le C\int_0^t\E\sup_{q\le s}\|Z_q\|^2ds
        +C\int_0^t\E\| (D_N^{1/2}-D_{\mathrm{eff}}^{1/2})Q^{1/2}\|_{\LLtwo}^2ds .
\]
The last integral is of order \(N^{-2\rho}\) by the assumption and the moment
estimate. Gronwall gives the mean-square estimate, and the inequality
\(W_2^2\le\E\|Z_t\|^2\) completes the proof.
\end{proof}

\section{The Commutative Benchmark and the Noncommutative Correction}
\label{sec:comm-noncomm}

For a high-level review it is important to separate two facts: when
\([D,Q]=0\) the consistent form coincides with the familiar diagonal
geometry, and when \([D,Q]\ne0\) the order of operators becomes part of the
mathematical result rather than a matter of notation. We record this below in
minimal form.

\begin{table}[h!]
\centering
\caption{The commutative and noncommutative regimes: what controls the dissipation.}
\label{tab:comm-noncomm-comparison}
\begin{tabular}{|p{0.28\textwidth}|p{0.30\textwidth}|p{0.30\textwidth}|}
\hline
Property & Commutative case & Noncommutative case \\
\hline
Operator order & \(Q^{1/2}D^{1/2}=D^{1/2}Q^{1/2}\) & order fixed by the generator: \(\sigma^*=Q^{1/2}D^{1/2}\) \\
\hline
Energy & \(\langle DQ\nabla u,\nabla u\rangle\) & \(\|Q^{1/2}D^{1/2}\nabla u\|^2\) \\
\hline
Constants & \(c_A=\inf\lambda(D)\), \(C_A=\sup\lambda(D)\) & extrema of a generalized problem in the \(Q\)-geometry \\
\hline
Risk of degeneracy & spectral & control may be lost in the directions \(Q^{1/2}D^{1/2}\xi\) \\
\hline
\end{tabular}
\end{table}

\begin{proposition}[The commutative benchmark]
\label{prop:commutative-benchmark}
Let \(D\) and \(Q\) be self-adjoint, positive, and commuting. Then for all
cylindrical \(u\),
\[
        \Gamma_A(u,u)
        =\|Q^{1/2}D^{1/2}\nabla u\|^2
        =\langle DQ\nabla u,\nabla u\rangle .
\]
If, in addition, \(d_-I\le D\le d_+I\), then
\[
        d_-\|Q^{1/2}\nabla u\|^2
        \le \Gamma_A(u,u)
        \le d_+\|Q^{1/2}\nabla u\|^2 .
\]
\end{proposition}

\begin{proof}
Commutativity gives \(D^{1/2}Q^{1/2}=Q^{1/2}D^{1/2}\), and hence
\[
        (Q^{1/2}D^{1/2})^*(Q^{1/2}D^{1/2})=D^{1/2}QD^{1/2}=DQ .
\]
The two-sided estimate follows from the spectral order for \(D\).
\end{proof}

\begin{proposition}[The noncommutative correction]
\label{prop:noncomm-correction}
Let \([D,Q]\ne0\). Then the expressions
\[
        \|Q^{1/2}D^{1/2}\xi\|^2,
        \qquad
        \|D^{1/2}Q^{1/2}\xi\|^2
\]
are in general different. Therefore the form consistent with the SDE
\(\sigma=D^{1/2}Q^{1/2}\) is defined through \(\sigma^*=Q^{1/2}D^{1/2}\), and
not through the transposed expression.
\end{proposition}

\begin{proof}
The carr\'e du champ of a diffusion with coefficient \(\sigma\) equals
\[
        \Gamma(u,u)=\|\sigma^*\nabla u\|^2.
\]
For \(\sigma=D^{1/2}Q^{1/2}\) we have \(\sigma^*=Q^{1/2}D^{1/2}\). If this
expression is replaced by \(D^{1/2}Q^{1/2}\), one obtains the form
corresponding to a different noise operator \(Q^{1/2}D^{1/2}\). Under
noncommutativity these operators are distinct, and hence so are the associated
forms.
\end{proof}

\begin{remark}[Interpretation]
Propositions~\ref{prop:commutative-benchmark}--\ref{prop:noncomm-correction}
show that the new geometry is not a cosmetic relabeling of the classical case.
In the commutative limit it returns the usual diagonal theory; in the
noncommutative case it selects the unique order of operators compatible with
the generator of the forward SDE and with the backward entropy identity. The
numerical test of Part~II gives a concrete illustration: for
\(N=128\), for the chosen noncommutative cell, \(\min d_i\simeq1.00\),
\(\max d_i\simeq2.10\), whereas the operator constants from Table~1 of the
second part are \(c_{A,N}=1.001\) and \(C_{A,N}=2.231\). Hence the
\(A\)-constants cannot be replaced by the Euclidean spectrum of the tensor
alone.
\end{remark}

\section{Conclusion to Part~I}

Part~I builds the analytic framework: the consistent \(A\)-form, the Galerkin
approximation, Mosco stability, the weak bridge, and the entropy chain rule.
The resulting estimate is an analytic reduction of the stability problem to
the control of the validation error \(\Eval\). The latter is treated in Part~II, together with empirical processes, lower bounds, and applications.

\bibliographystyle{unsrt}
\bibliography{references}

\end{document}